\documentclass[a4paper,12pt]{article}
\usepackage{amsmath,amsthm,amsfonts,amssymb,bbm}
\usepackage{graphicx,psfrag,subfigure}
\usepackage{cite}
\usepackage{hyperref}
\usepackage[hmargin=1in,vmargin=1in]{geometry}
\hypersetup{colorlinks=true}

\renewcommand{\d}{\mathrm d}
\newcommand{\R}{\mathbb R}
\newcommand{\Z}{\mathbb Z}

\newcommand{\wt}{\widetilde}
\newcommand{\wh}{\widehat}

\renewcommand{\Re}{\operatorname{Re}}

\newcommand{\sgn}{\operatorname{sgn}}

\newcommand{\C}{\mathcal C}
\newcommand{\D}{\mathcal D}

\newcommand{\id}{\mathbbm{1}}
\renewcommand{\O}{\mathcal{O}}
\renewcommand{\P}{\mathbf P}

\newtheorem{proposition}{Proposition}[section]
\newtheorem{theorem}[proposition]{Theorem}

\newtheorem{corollary}[proposition]{Corollary}
\newtheorem{conjecture}[proposition]{Conjecture}
\theoremstyle{definition}
\newtheorem{remark}[proposition]{Remark}

\numberwithin{equation}{section}

\title{Fluctuations around the diagonal in Bernoulli-Exponential first passage percolation}
\author{B\'alint Vet\H o\thanks{Department of Stochastics, Institute of Mathematics,
Budapest University of Technology and Economics, M\H uegyetem rkp.\ 3., H-1111 Budapest, Hungary. E-mail: {\tt vetob@math.bme.hu}}
\thanks{ELKH--BME Stochastics Research Group, M\H uegyetem rkp.\ 3., H-1111 Budapest, Hungary}}

\begin{document}

\maketitle

\begin{abstract}
We prove that the rescaled one-point fluctuations of the boundary of the percolation cluster
in the Bernoulli-Exponential first passage percolation around the diagonal converge to a new family of distributions.
The limit law is indexed by the rescaled level of percolation $s\ge0$,
it is Gaussian for $s=0$ and it converges to the Tracy--Widom distribution as $s\to\infty$.
For a fixed level $s>0$ the width of the cluster in the limit as a function of a time parameter $t$
is of order $t^{2/3}$ with Tracy--Widom fluctuations as in the discrete model.
\end{abstract}

\section{Introduction}

The Bernoulli-Exponential directed first passage percolation was introduced in~\cite{BC17} as follows.
Let $a,b>0$ be fixed.
Let $(E_e)$ be a family of independent random variables indexed by the edges of the lattice $\Z^2$
where the distribution of $E_e$ is exponential with parameter $a$ if $e$ is a vertical edge and exponential with parameter $b$ is $e$ is a horizontal edge.
Let $(\xi_{i,j})$ be independent Bernoulli random variables with parameter $b/(a+b)$ which are also independent of $(E_e)$.
The passage times of edges are given by
\begin{equation}
t_e=\left\{\begin{array}{cl} \xi_{i,j}E_e & \mbox{if $e$ is the vertical edge $(i,j)\to(i,j+1)$},\\
(1-\xi_{i,j})E_e & \mbox{if $e$ is the horizontal edge $(i,j)\to(i+1,j)$.}\end{array}\right.
\end{equation}
For non-negative integers $n$ and $m$ the point-to-point first passage time is given by
\begin{equation}
T^{\rm pp}(n,m)=\min_{\pi:(0,0)\to(n,m)}\sum_{e\in\pi}\,t_e
\end{equation}
where the minimum is over all up-right paths $\pi$ from $(0,0)$ to $(n,m)$.

We also introduce the point to half-line first passage time $T(n,m)$ between $(0,0)$ and the half-line
\begin{equation}
D_{n,m}=\{(i,n+m-i):0\le i\le n\}
\end{equation}
to be given by
\begin{equation}
T(n,m)=\min_{\pi:(0,0)\to D_{n,m}}\sum_{e\in\pi}\,t_e
\end{equation}
where the minimum is taken over all up-right paths $\pi$ from $(0,0)$ to $D_{n,m}$.

It was proved in~\cite{BC17} that for any slope $\kappa>a/b$,
the fluctuations of the passage time $T(n,\kappa n)$ converges to the GUE Tracy--Widom distribution,
but the behaviour around the slope $a/b$ was not considered.
These results were extended in~\cite{BR19} with a theorem about the GUE Tracy--Widom fluctuations of $T(n,an/b+cn^{2/3})$ for any $c>0$.

In this note we investigate the asymptotic fluctuations of the passage time when approaching the diagonal of slope $a/b$ on the scale $\sqrt n$
on which a new family of distribution arises in the limit.
The asymptotic fluctuations around the diagonal can be expressed in two equivalent ways.
We state the main result in Theorem~\ref{thm:BEasymptotics} in terms of the shape of the percolation cluster.
In Corollary~\ref{cor:FPP} we explicitly write the fluctuations of the first passage time value $T(n,an/b+cn^{1/2})$.

For any level $r\ge0$ the percolation cluster is defined by
\begin{equation}
C(r)=\{(n,m):T^{\rm pp}(n,m)\le r\}.
\end{equation}
It is natural to introduce the height function
\begin{equation}\label{defH}
H(n,r)=\max\{k\in\Z:T^{\rm pp}(bn-k,an+k)\le r\}
\end{equation}
where $n$ is a non-negative integer and $r\ge0$.
Note that the maximum always exists on the right-hand side of \eqref{defH} for any $r\ge0$
because there is always a path from $(0,0)$ to $D_{(a+b)n,0}$ with zero first passage time value.
We state the main result in terms of the height function $H(n,r)$ as follows.

\begin{theorem}\label{thm:BEasymptotics}
Fix an $s>0$.
Then
\begin{equation}\label{BEasymptotics}
\sqrt{\frac{a+b}{ab}}\frac1{\sqrt n}H\left(n,\frac s{\sqrt{ab(a+b)}}n^{-1/2}\right)\stackrel{\d}{\Longrightarrow}H_s
\end{equation}
in distribution as $n\to\infty$ where the distribution of $H_s$ is given as follows.
For any $h\in\R$,
\begin{equation}\label{Hsdistribution}
\P(H_s<h)=\det(\id-K_s)_{L^2((h,\infty))}
\end{equation}
with the kernel
\begin{equation}\label{defKs}
K_s(x,y)=\frac1{(2\pi i)^2}\int_{1+i\R}\d u\int_{\C_0}\d v\,\frac{e^{u^2/2-yu-s/u}}{e^{v^2/2-xv-s/v}}\frac uv\frac1{v-u}
\end{equation}
where the integration contour $\C_0$ is a small circle around $0$ with positive orientation such that it does not intersect $1+i\R$.
\end{theorem}

\begin{remark}
The formal substitution $s=0$ in \eqref{Hsdistribution}--\eqref{defKs} yields the standard Gaussian distribution.
It can be seen by observing that the $v$-integral is equal to the residue at $v=0$
and by computing the $u$-integral directly to get that $K_0(x,y)=\frac1{\sqrt{2\pi}}e^{-y^2/2}$.
This corresponds to taking the limit of $\sqrt{\frac{a+b}{ab}}\frac1{\sqrt n}H(n,0)$ which is not covered by the statement of Theorem~\ref{thm:BEasymptotics},
but this limit is known to be Gaussian since it is the scaling limit of a simple random walk with Bernoulli steps.
\end{remark}

\begin{theorem}\label{thm:HstoTW}
The rescaled random variables
\begin{equation}\label{HstoTW}
2^{4/9}3^{-1/3}s^{1/9}\left(H_s-2^{-2/3}3s^{1/3}\right)\stackrel\d\Longrightarrow\xi
\end{equation}
as $s\to\infty$ where $\xi$ has GUE Tracy--Widom distribution.
\end{theorem}

\begin{corollary}\label{cor:clusterwidth}
For a fixed $s>0$ we introduce the height of the percolation cluster of level $s$ after time $t>0$ to be
\begin{equation}\label{defHst}
H_s(t)=\lim_{n\to\infty}\sqrt{\frac{a+b}{ab}}\frac1{\sqrt n}H\left(tn,\frac s{\sqrt{ab(a+b)}}n^{-1/2}\right).
\end{equation}
For any $s>0$ the rescaled cluster height converges, that is,
\begin{equation}\label{Hstconv}
\frac{H_s(t)-2^{-2/3}3s^{1/3}t^{2/3}}{2^{-4/9}3^{1/3}s^{-1/9}t^{4/9}}\stackrel\d\Longrightarrow\xi
\end{equation}
as $t\to\infty$ where $\xi$ has GUE Tracy--Widom distribution.
\end{corollary}

\begin{remark}
The limit in \eqref{defHst} exists by Theorem~\ref{thm:BEasymptotics} for any fixed $t>0$.
Corollary~\ref{cor:clusterwidth} does not imply the existence of the time process $t\mapsto H_s(t)$.
We expect that the limit process in \eqref{defHst} can be constructed as a function of $t$ based on the Brownian web,
see~\cite{TW98,FINR04}.

By the results of~\cite{BR19}, the width of the percolation cluster of a fixed level in the Bernoulli-Exponential model
along the diagonal $an/b$ is of order $n^{2/3}$ with Tracy--Widom fluctuations on the scale $n^{4/9}$.
By Corollary~\ref{cor:clusterwidth}, the height of the cluster in the limit as a function of $t$
has the same limiting fluctuations under the same scaling as in the discrete model.
\end{remark}

\begin{remark}
The kernel $K_s$ in \eqref{defKs} is reminiscent of the correlation kernel of the hard-edge Pearcey process
which arises in the neighbourhood of the cusp point of the limit shape in the situation when non-intersecting paths are pushed towards a hard wall.
In the case of non-intersecting squared Bessel paths,
the single-time kernel of the limit process was first described in~\cite{KMW11} and the multi-time kernel was given in~\cite{DV15}.
We describe the connection of the two kernels below in more details.
Let
\begin{equation}
L_s(x,y)=\frac1{(2\pi i)^2}\int_{\C_0}\d w\int_{1+i\R}\d z\,\frac1{wz(w-z)}\frac{e^{-w^2/2+sw+x/w}}{e^{-z^2/2+sz+y/z}}
\end{equation}
be the single-time kernel of the hard-edge Pearcey process.
It was given in a slightly different form in Theorem 1.2 of~\cite{KMW11} with $\alpha=-1$ and more explicitly up to a conjugation in Proposition 2.21 of~\cite{DV15} with $t=s$, $\alpha=-1$ and $\sigma=0$.
Here $\alpha$ denotes the index of the squared Bessel paths which is assumed to be $\alpha>-1$ in~\cite{KMW11,DV15}, hence the substitution $\alpha=-1$ is formal.
\end{remark}

\begin{proposition}
The derivative of the kernel $K_s$ and that of $L_s$ with respect to $s$ factorize as
\begin{align}
\frac{\d}{\d s}K_s(x,y)&=f(s,x)g(s,y),\\
-\frac{\d}{\d s}L_s(x,y)&=f(x,s)g(y,s)
\end{align}
where
\begin{align}
f(s,x)&=\frac1{2\pi i}\int_{\C_0}\frac{\d v}{v^2}\,e^{-v^2/2+xv+s/v},\\
g(s,y)&=\frac1{2\pi i}\int_{1+i \R}\d u\,e^{u^2/2-yu-s/u}.
\end{align}
\end{proposition}

The upper tail decay of the random variables $H_s$ is close to Gaussian.

\begin{proposition}\label{prop:Hstail}
\begin{enumerate}
\item
There is a universal constant $C$ and a threshold $h_0>0$ such that we have
\begin{equation}\label{Hstail}
\P(H_s>h)\le C\frac{e^{-h^2/2+4\sqrt{sh}}}h
\end{equation}
for all $h\ge h_0$ if $0\le s\le h$ holds.
\item
If both $h,s\to\infty$ in a way that $s\ll h^3$,
then the tail bound in \eqref{Hstail} remains valid with the factor $4$ in the exponent is replaced by $2+o(1)$ as $h\to\infty$.
\item
There is a $c^*\simeq0.0468$ such that if $s=ch^3$ with $c\in(0,c^*)$, then $\P(H_s>h)\le e^{-\delta(c)h^2}$ as $h\to\infty$ with some $\delta(c)>0$.
\end{enumerate}
\end{proposition}

Theorem~\ref{thm:BEasymptotics} can be translated into a fluctuation result on the passage times as follows.
It is a direct consequence of the definition \eqref{defH} of the height function $H(n,r)$ that
\begin{equation}\label{eqevents}
\{T(bn-k,an+k)>r\}=\{H(n,r)<k\}.
\end{equation}
This equality of events yields the following result on the passage times.

\begin{corollary}\label{cor:FPP}
Let $h\in\R$ be fixed.
Then
\begin{equation}
\sqrt{ab(a+b)}\sqrt n\,T\left(bn-\sqrt{\frac{ab}{a+b}}h\sqrt n,an+\sqrt{\frac{ab}{a+b}}h\sqrt n\right)\stackrel{\d}{\Longrightarrow}T_h
\end{equation}
in distribution as $n\to\infty$.
The distribution of $T_h$ has an atom at $0$ with weight
\begin{equation}
\P(T_h=0)=\int_h^\infty\frac1{\sqrt{2\pi}}e^{-y^2/2}\d y.
\end{equation}
The distribution function of $T_h$ for any $s>0$ is given by
\begin{equation}\label{Thdistribution}
\P(T_h>s)=\det(\id-K_s)_{L^2((h,\infty))}
\end{equation}
where the kernel $K_s$ is defined in \eqref{defKs}.
\end{corollary}

The Tracy--Widom limit of $H_s$ in Theorem~\ref{thm:HstoTW} implies a similar result for the limiting passage times.

\begin{corollary}\label{cor:passagetimetoTW}
For the rescaled limiting passage time it holds that
\begin{equation}\label{passagetimetoTW}
\left(\frac32\right)^{4/3}h^{-5/3}\left(\frac{4h^3}{27}-T_h\right)\stackrel\d\Longrightarrow\xi
\end{equation}
as $h\to\infty$ where $\xi$ has GUE Tracy--Widom distribution.
\end{corollary}

We expect that the Tracy--Widom limit of the passage time extends to the following convergence to the Airy process.
The scaling of the space variable $x$ in \eqref{TtoAiry} below can be guessed based on the Taylor expansion of the limit shape in \eqref{passagetimetoTW}.

\begin{conjecture}\label{conj:TtoAiry}
In the parameter $x\in\R$ we have that
\begin{equation}\label{TtoAiry}
\left(\frac32\right)^{4/3}h^{-5/3}\left(\frac{4h^3}{27}+\left(\frac23\right)^{5/3}h^{7/3}x-T_{h+(3/2)^{1/3}h^{1/3}x}\right)\stackrel\d\Longrightarrow\mathcal A(x)-x^2
\end{equation}
as $h\to\infty$ where $\mathcal A(x)$ is the stationary Airy process.
\end{conjecture}

The rest of this note is organized as follows.
In Section~\ref{s:reformulation}, we reformulate the Fredholm determinant expression from~\cite{BC17}
for the point to half-line first passage time in the Bernoulli-Exponential model.
We prove Theorem~\ref{thm:BEasymptotics} the main result in this note in Section~\ref{s:asymptotics}
which is based on some asymptotic statements proved in Section~\ref{s:asympproofs}.
We prove the Tracy--Widom fluctuations in the $s\to\infty$ limit in Section~\ref{s:TW}
and the decay bounds of Proposition~\ref{prop:Hstail} in Section~\ref{s:decaybounds}.

\paragraph{Acknowledgements:}
We thank B\'alint Vir\'ag and Patrik Ferrari for discussions about polymer models and correlation kernels
and Guillaume Barraquand for pointing out the scaling in Theorem~\ref{thm:HstoTW} and for his comments.
The work of the author was supported by the NKFI (National Research, Development and Innovation Office)
grants FK142124 and KKP144059 ``Fractal geometry and applications'',
by the Bolyai Research Scholarship of the Hungarian Academy of Sciences
and by the \'UNKP--22--5--BME--250 New National Excellence Program of the Ministry for Innovation and Technology
from the source of the NKFI.

\section{Reformulation of the passage time distribution}
\label{s:reformulation}

The distribution of the point to half-line Bernoulli-Exponential first passage time is characterized by the following result
which is based on Theorem 1.18 of~\cite{BC17} taking into account Remark 1.6 in~\cite{BR19} about a correct sign in \eqref{BEprobBC} below.

\begin{theorem}\label{thm:BEfinite}
Let $r>0$ and let $n,m$ be non-negative integers.
Then for the point to half-line Bernoulli-Exponential first passage time $T(n,m)$ with parameters $a,b>0$, we have
\begin{equation}\label{BEprobBC}
\P(T(n,m)>r)=\det(\id-\wh K_r)_{L^2(\C'_0)}
\end{equation}
where $\C'_0$ is a small positively oriented circle around $0$ not containing $-a-b$, and the kernel is given by
\begin{equation}\label{defKrhat}
\wh K_r(u,u')=\frac1{2\pi i}\int_{\frac12+i\R}\frac{e^{rs}}s\frac{\wh g(u)}{\wh g(u+s)}\frac{\d s}{s+u-u'}
\end{equation}
with
\begin{equation}\label{defghat}
\wh g(u)=\left(\frac{a+u}u\right)^n\left(\frac{a+u}{a+b+u}\right)^m\frac1u.
\end{equation}
\end{theorem}

We reformulate the statement of Theorem~\ref{thm:BEfinite} by a change of variables as follows.

\begin{proposition}\label{prop:BEalternative}
Let $n$ be a non-negative integer and $r>0$.
Then for any $k\in\Z$,
\begin{equation}\label{BEprob}
\P(H(n,r)<k)=\det(\id+\wt K_r)_{L^2(\C_{-1/(a+b)})}
\end{equation}
where the kernel is given by
\begin{equation}
\wt K_r(u,u')=\frac1{2\pi i}\int_{\D_0}\frac{e^{r(1/v-1/u)}}{(v-u)(v-u')}\frac{g(u)}{g(v)}\,\d v
\end{equation}
and
\begin{equation}\label{defg}
g(u)=\frac{(1+au)^{(a+b)n}}{(1+(a+b)u)^{an+k}}u.
\end{equation}
The contour $\D_0$ is a circle around $0$ not containing $-1/(a+b)$ and $\C_{-1/(a+b)}$ is a large contour which encircles $\D_0$ and $-1/(a+b)$.
\end{proposition}

\begin{proof}
We use the statement of Theorem~\ref{thm:BEfinite} to derive \eqref{BEprob}.
The left-hand side of \eqref{BEprobBC} and that of \eqref{BEprob} are equal due to the equality of the events \eqref{eqevents}.
The equality of the right-hand sides follows in the steps given below.
First note that the integration over $1/2+i\R$ is formal in \eqref{defKrhat} because of the oscillatory behaviour of the integrand.
One way how it can be understood is to integrate over the contour
\begin{equation}
\D_R=\left\{1/2+iy:y\in[-R,R]\right\}\cup\left\{1/2+Re^{i\phi}:\phi\in\left[\pi/2,3\pi/2\right]\right\}.
\end{equation}

Then we rewrite the integral in \eqref{defKrhat} in terms of the variables $v=u+s$ over the same integration contour $\D_R$ as follows
\begin{equation}
\wh K_r(u,u')=\frac1{2\pi i}\int_{\D_R}\frac{e^{r(v-u)}}{v-u}\frac{\wh g(u)}{\wh g(v)}\frac{\d v}{v-u'}.
\end{equation}
The main step is the change of variables $u\to1/u$, $u'\to1/u'$, $v\to1/v$.
It yields the equality of Fredholm determinants $\det(\id-\wh K_r)_{L^2(\C'_0)}=\det(\id+\wt K_r)_{L^2(\C_{-1/(a+b)})}$ with
\begin{equation}
\wt K_r(u,u')=\frac1{uu'}\wh K_r\left(\frac1u,\frac1{u'}\right)
=\frac1{2\pi i}\int_{\D_0}\frac{e^{r(1/v-1/u)}}{v-u}\frac{\wh g(1/u)}{\wh g(1/v)}\frac{\d v}{v-u'}
\end{equation}
with the contours $\C_{-1/(a+b)}$ and $\D_0$ defined below \eqref{defg}.
The sign change of the kernel is due to the orientation of the contours.
Then \eqref{BEprob} follows by comparing the definition \eqref{defghat} with $(n,m)$ replaced by $(bn-k,an+k)$ and \eqref{defg}.
\end{proof}

\section{Asymptotic analysis}
\label{s:asymptotics}

This section is devoted to the proof of Theorem~\ref{thm:BEasymptotics} which is the main result in this note.
The technical proofs of Propositions~\ref{prop:steep}, \ref{prop:localization}, \ref{prop:Taylor} and \ref{prop:convergence} about specific parts of the asymptotics
are postponed to Section~\ref{s:asympproofs}.

With the notation
\begin{equation}
s_n=\frac s{\sqrt{ab(a+b)}}n^{-1/2},\qquad h_n=\sqrt{\frac{ab}{a+b}}hn^{1/2}
\end{equation}
the convergence result \eqref{BEasymptotics} can be written as
\begin{equation}
\lim_{n\to\infty}\P(H(n,s_n)<h_n)=\P(H_s<h).
\end{equation}

By Proposition~\ref{prop:BEalternative}, we have that
\begin{equation}\label{Hprobscaled}
\P(H(n,s_n)<h_n)=\det(\id+\wt K_{s_n})_{L^2(C_{-1/(a+b)})}
\end{equation}
where the kernel can be given as
\begin{equation}\label{defKsntilde}
\wt K_{s_n}(u,u')=\frac1{2\pi i}\int_{\D_0}e^{n(f_0(u)-f_0(v))+\sqrt n(f_1(u)-f_1(v))+s_n(\frac1v-\frac1u)}\,\frac uv\frac{\d v}{(v-u)(v-u')}
\end{equation}
with
\begin{align}
f_0(u)&=(a+b)\ln(1+au)-a\ln(1+(a+b)u),\\
f_1(u)&=-\sqrt{\frac{ab}{a+b}}h\ln(1+(a+b)u).
\end{align}
Hence the proof of Theorem~\ref{thm:BEasymptotics} boils down to show the convergence of the Fredholm determinants
\begin{equation}
\lim_{n\to\infty}\det(\id+\wt K_{s_n})_{L^2(C_{-1/(a+b)})}=\det(\id-K_s)_{L^2((h,\infty))}.
\end{equation}

Since
\begin{equation}
f_0'(u)=\frac{ab(a+b)u}{(1+au)(1+(a+b)u)},
\end{equation}
the function $f_0(u)$ has a unique critical point at $0$.
Its Taylor expansion around this point is
\begin{equation}\label{f0Taylor}
f_0(u)=\frac12ab(a+b)u^2+\O(u^3)
\end{equation}
as $u\to0$.
The first step of the asymptotic analysis is to find contours which enable us to localize the contour
on which the Fredholm determinant is defined as well as the integration in \eqref{defKsntilde} to a neighbourhood of $0$.
The existence of appropriate contours is ensured by Proposition~\ref{prop:steep} below.
We introduce the V-shaped contour
\begin{equation}
V_{\alpha,\varphi}^\delta=\{\alpha+e^{i\varphi\sgn(t)}|t|:t\in[-\delta,\delta]\}
\end{equation}
where $\alpha\in\C$ is the tip of the V, $\varphi\in(0,\pi)$ is its half-angle and $\delta\in\R_+\cup\{\infty\}$ is its length.

\begin{proposition}\label{prop:steep}
There exist two bounded closed contours $\gamma_\pm$ such that $\Re(f_0(v))\ge0$ for $v\in\gamma_+$ and $\Re(f_0(u))\le0$ for $u\in\gamma_-$.
Moreover, for a small $\delta>0$,
\begin{equation}\label{gammapmdef}
\gamma_+\cap B(0,\delta)=V_{0,5\pi/6}^\delta,\qquad\gamma_-\cap B(0,\delta)=V_{0,\pi/2}^\delta
\end{equation}
where $B(0,\delta)$ denotes the ball of radius $\delta$ around $0$.
As a consequence, for any $\varepsilon>0$ small enough there is a $\delta'>0$ such that $\Re(f_0(u))<-\varepsilon$ for $u\in\gamma_-\setminus B(0,\delta')$
and $\Re(f_0(v))>\varepsilon$ for $v\in\gamma_+\setminus B(0,\delta')$.
\end{proposition}

A possible choice of these contours is shown on Figure~\ref{fig:level_lines}.
Let $\gamma_+^n$ be equal to $\gamma_+$ of Proposition~\ref{prop:steep} except for an $n^{-1/2}$ neighbourhood of $0$ where $\gamma_+^n$ is defined to be
\begin{equation}
\gamma_+^n\cap B(0,n^{-1/2})=\{n^{-1/2}e^{i\varphi}:\varphi\in[-5\pi/6,5\pi/6]\}
\end{equation}
for $n$ large enough.
Let $\gamma_-^n$ be equal to $\gamma_-$ of Proposition~\ref{prop:steep} except for a $2n^{-1/2}$ neighbourhood of $0$ where $\gamma_-^n$ is defined to be
\begin{equation}
\gamma_-^n\cap B(0,2n^{-1/2})=\{2n^{-1/2}e^{i\varphi}:\varphi\in[-\pi/2,\pi/2]\}
\end{equation}
for $n$ large enough.
Then the contours used on the right-hand side of \eqref{Hprobscaled} can be replaced by $\gamma_\pm$ as follows.
By Cauchy's integral theorem, we can deform the contour $\C_{-1/(a+b)}$ to $\gamma_-^n$ on the right-hand side of \eqref{Hprobscaled}.
The integration contour in the formula \eqref{defKsntilde} for the kernel $\wt K_{s_n}$ can also be deformed to $\gamma_+^n$
without changing the value of the Fredholm determinant.
Note that there is no singularity in the variable $v$ at $-1/(a+b)$.

Next we localize the integration to a neighbourhood of $0$ on the right-hand side of \eqref{Hprobscaled}.
For $\delta>0$, let
\begin{equation}
\gamma_\pm^{n,\delta}=\gamma_\pm^n\cap B(0,\delta)
\end{equation}
denote the contours $\gamma_\pm^n$ restricted to the $\delta$-neighbourhood of $0$.
We define the kernel
\begin{equation}\label{Krtildedelta}
\wt K_{s_n}^\delta(u,u')=\frac1{2\pi i}\int_{\gamma_+^{n,\delta}}e^{n(f_0(u)-f_0(v))+\sqrt n(f_1(u)-f_1(v))+s_n(\frac1v-\frac1u)}\,
\frac uv\frac{\d v}{(v-u)(v-u')}
\end{equation}
which differs from $\wt K_{s_n}$ given in \eqref{defKsntilde} only in the choice of the integration contour.
The Fredholm determinant in \eqref{Hprobscaled} and that of \eqref{Krtildedelta} over the sequence of contours $\gamma_-^{n,\delta}$ have the same limit,
that is, the localization does not change the $n\to\infty$ limit.

\begin{proposition}\label{prop:localization}
For any $\delta>0$ small enough, we have that
\begin{equation}\label{localization}
\lim_{n\to\infty}\det(\id+\wt K_{s_n})_{L^2(\gamma_-^n)}=\lim_{n\to\infty}\det(\id+\wt K_{s_n}^\delta)_{L^2(\gamma_-^{n,\delta})}.
\end{equation}
\end{proposition}

The next statement is about the Taylor expansion of the localized Fredholm determinant.

\begin{proposition}\label{prop:Taylor}
For $\delta>0$ small enough, the following limits are equal
\begin{equation}\label{deteqTaylor}
\lim_{n\to\infty}\det(\id+\wt K_{s_n}^\delta)_{L^2(\gamma_-^{n,\delta})}=\lim_{n\to\infty}\det(\id+K'_{s,n})_{L^2(\Gamma_n')}
\end{equation}
where
\begin{equation}\label{defK'}
K'_{s,n}(U,U')=\frac1{2\pi i}\int_{\Gamma_n}\frac{e^{U^2/2-hU-s/U}}{e^{V^2/2-hV-s/V}}\frac UV\frac{\d V}{(V-U)(V-U')}.
\end{equation}
The integration contour $\Gamma_n=\Gamma\cap B(0,\sqrt{ab(a+b)n}\delta)$
where $\Gamma$ is a path from $e^{-5\pi i/6}\infty$ to $e^{5\pi i/6}\infty$ so that it crosses the real axis between $0$ and $1$.
The contour $\Gamma_n'$ is the vertical segment between $\pm i\sqrt{ab(a+b)n}\delta$ oriented upwards
and modified around $0$ so that it does not intersect $\Gamma_n$.
\end{proposition}

Finally, the proposition below yields the convergence of the localized Fredholm determinant to the right-hand side of \eqref{Hsdistribution}.
That is, Theorem~\ref{thm:BEasymptotics} follows from Propositions~\ref{prop:localization}, \ref{prop:Taylor} and \ref{prop:convergence}.

\begin{proposition}\label{prop:convergence}
Let $\delta>0$ be small.
Then
\begin{equation}\label{K'conv}
\lim_{n\to\infty}\det(\id+K'_{s,n})_{L^2(\Gamma_n')}=\det(\id-K_s)_{L^2((h,\infty))}.
\end{equation}
\end{proposition}

\section{Proofs of the asymptotic statements}
\label{s:asympproofs}

In this section we prove the asymptotic statements used in the proof of Theorem~\ref{thm:BEasymptotics}.

\begin{proof}[Proof of Proposition~\ref{prop:steep}]
Since $f_0$ is analytic away from its singularities, $\Re(f_0)$ is harmonic and its level lines of the form $\Re(f_0(u))=0$ can be described as follows.
The level lines can only cross at singularities or critical points.
There are two singularities of $f_0$ at $-1/a$ and at $-1/(a+b)$ and a critical point at $0$.
It follows from the Taylor expansion \eqref{f0Taylor} that the branches of the level lines $\Re(f_0(u))=0$ cross at $0$ with angles $\pm\pi/4$ and $\pm3\pi/4$.
As $|u|\to\infty$ in any direction, $\Re(f_0(u))\to\infty$, hence all level lines remain bounded.
By the maximum principle, any closed path formed by portions of level lines must enclose a singularity.
Around the singularity at $-1/a$, $\Re(f_0)$ is negative and around $-1/(a+b)$, $\Re(f_0)$ is positive.

Based on this information, the only possible configuration of the level lines $\Re(f_0)=0$
up to a continuous deformation of the lines which does not cross any singularity is shown on Figure~\ref{fig:level_lines}.
Then the contours $\gamma_\pm$ are defined in two steps.
We first choose a small $\delta>0$ and give $\gamma_\pm$ in $B(0,\delta)$ to be defined by \eqref{gammapmdef}
and we let the value of $\Re(f_0)$ at the endpoints be denoted by $\varepsilon_+=\Re(f_0(e^{\pm5\pi i/6}\delta))>0$ and $\varepsilon_-=\Re(f_0(\pm i\delta))<0$.
Then in the second step, we define $\gamma_+$ outside of $B(0,\delta)$ to coincide with that branch of the level line $\Re(f_0(v))=\varepsilon_+$
which connects the two points $e^{\pm5\pi i/6}\delta$.
Similarly, we let $\gamma_-$ outside of $B(0,\delta)$ to be the same as the branch of the level line $\Re(f_0(u))=\varepsilon_-$
which connects the points $\pm i\delta$.
Then the contours $\gamma_\pm$ satisfy the required properties by the Taylor expansion \eqref{f0Taylor}.
\begin{figure}
\centering
\psfrag{ab}{$-\frac1{a+b}$}
\psfrag{a}{$-\frac1a$}
\psfrag{gp}{$\gamma_+$}
\psfrag{gm}{$\gamma_-$}
\includegraphics[width=200pt]{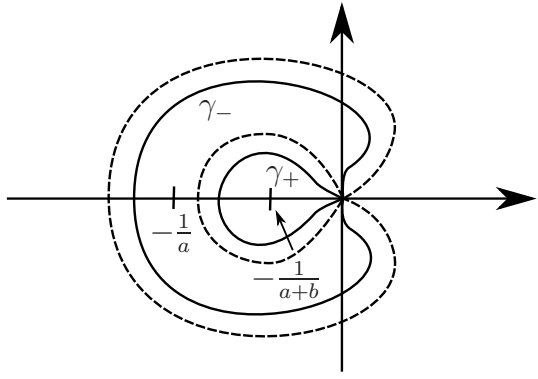}
\caption{The level lines $\Re(f_0)=0$ shown by dashed lines and a possible choice of the integration contours $\gamma_\pm$ shown by solid lines.
\label{fig:level_lines}}
\end{figure}
\end{proof}

\begin{proof}[Proof of Proposition~\ref{prop:localization}]
We fix a small $\delta>0$.
The integrand in \eqref{Krtildedelta} can be upper bounded as
\begin{equation}\label{Krtildeintegrand}
\left|e^{n(f_0(u)-f_0(v))+\sqrt n(f_1(u)-f_1(v))+s_n(\frac1v-\frac1u)}\,\frac uv\frac1{(v-u)(v-u')}\right|\le Cn^{3/2}e^{n(\Re(f_0(u)-f_0(v)))+K\sqrt n}
\end{equation}
for $u,u'\in\gamma_-^n$ and $v\in\gamma_+^n$ with some finite constants $C,K$.

By Proposition~\ref{prop:steep}, $\Re(f_0(v))>\varepsilon$ for all $v\in\gamma_+^n\setminus\gamma_+^{n,\delta}$ with some $\varepsilon>0$.
Hence the integration over $\gamma_+^n\setminus\gamma_+^{n,\delta}$ can be upper bounded as
\begin{equation}
\left|\wt K_{s_n}(u,u')-\wt K_{s_n}^\delta(u,u')\right|\le e^{-n\varepsilon/2}
\end{equation}
for $n$ large enough for all $u,u'\in\gamma_-^n$.

Next we consider the Fredholm expansion
\begin{equation}
\det(\id+\wt K_{s_n})_{L^2(\gamma_-^n)}
=\sum_{k=0}^\infty\frac1{k!}\int_{\gamma_-^n}\d u_1\dots\int_{\gamma_-^n}\d u_k\det\left(\wt K_{s_n}(u_i,u_j)\right)_{i,j=1}^k.
\end{equation}
The integration in the $k$th term of the expansion can be written as the sum of the integral over $(\gamma_-^{n,\delta})^k$
and the integral over $(\gamma_-^n)^k\setminus(\gamma_-^{n,\delta})^k$.
By Proposition~\ref{prop:steep}, $\Re(f_0(u))<-\varepsilon$ for all $u\in\gamma_-^n\setminus\gamma_-^{n,\delta}$,
hence we can use the bound in \eqref{Krtildeintegrand} to conclude that the total contribution of the integrals over $(\gamma_-^n)^k\setminus(\gamma_-^{n,\delta})^k$
for $k=0,1,2,\dots$ goes to $0$ as $n\to\infty$.
In the integral over $(\gamma_-^{n,\delta})^k$, we can write the kernel $\wt K_{s_n}(u,u')$ as $\wt K_{s_n}^\delta(u,u')$ plus an error at most $e^{-n\varepsilon/2}$.
Hence the difference of the contribution over $(\gamma_-^{n,\delta})^k$ in the $k$th term of the Fredholm expansion of $\det(\id-\wt K_{s_n})_{L^2(\gamma_-^n)}$
and the $k$th term in the Fredholm expansion of $\det(\id+\wt K_{s_n}^\delta)_{L^2(\gamma_-^{n,\delta})}$ is summable in $k$ and the sum goes to $0$ as $n\to\infty$
proving \eqref{localization}.
\end{proof}

\begin{proof}[Proof of Proposition~\ref{prop:Taylor}]
By the Taylor expansion \eqref{f0Taylor} and
\begin{equation}
f_1(u)=-\sqrt{ab(a+b)}hu+\O(u^2)
\end{equation}
as $u\to0$, we can rewrite \eqref{Krtildedelta} as
\begin{multline}
\wt K_{s_n}^\delta(u,u')\\
=\frac1{2\pi i}\int_{\gamma_+^{n,\delta}}e^{n\frac12ab(a+b)(u^2-v^2)+\O(n(u^3+v^3))-\sqrt n\sqrt{ab(a+b)}h(u-v)+\O(\sqrt n(u^2+v^2))+s_n(\frac1v-\frac1u)}\\
\times\frac uv\frac{\d v}{(v-u)(v-u')}.
\end{multline}
By the change of variables $U=\sqrt n\sqrt{ab(a+b)}u$, $U'=\sqrt n\sqrt{ab(a+b)}u'$ and $V=\sqrt n\sqrt{ab(a+b)}v$, we get that the rescaled kernel is given by
\begin{multline}\label{Ksndeltawitherror}
\frac{n^{-1/2}}{\sqrt{ab(a+b)}}\wt K_{s_n}^\delta\left(\frac{n^{-1/2}}{\sqrt{ab(a+b)}}U,\frac{n^{-1/2}}{\sqrt{ab(a+b)}}V\right)\\
=\frac1{2\pi i}\int_{\Gamma_n}e^{U^2/2-V^2/2+\O(n^{-1/2}(U^3+V^3))-h(U-V)+\O(n^{-1/2}(U^2+V^2))+s/V-s/U}\\
\times\frac UV\frac{\d V}{(V-U)(V-U')}.
\end{multline}
The difference between the rescaled kernel above and $K'_{s_n}(U,U')$ in \eqref{defK'} is
the presence of the error terms $\O(n^{-1/2}(U^3+V^3))$ and $\O(n^{-1/2}(U^2+V^2))$ in the exponent.
Hence the integrand of the rescaled kernel above converges to that of $K'_{s_n}(U,U')$ for any $U,U'\in\Gamma_n'$ and $V\in\Gamma_n$.
In order to see the convergence of the kernels and that of the Fredholm determinants, we use dominated convergence.
We observe that along the integration contours the error terms can be bounded by a fixed constant times $\delta(U^2+V^2)$.
We bound the difference of the integrand with and without the error terms in the exponent by applying the inequality $|e^x-1|\le|x|e^{|x|}$.
The decay of the integrand in \eqref{Ksndeltawitherror} comes from the main term $e^{U^2/2-V^2/2}$, hence in the presence of the error terms bounded by $e^{C\delta(U^2+V^2)}$,
it remains integrable in both variables $U$ and $V$ if $\delta$ is small enough.
Hence the difference of the Fredholm determinants goes to $0$ as $n\to\infty$ by dominated convergence.
\end{proof}

\begin{proof}[Proof of Proposition~\ref{prop:convergence}]
The integrand in \eqref{defK'} has a Gaussian decay in both $U$ and $V$ due to the factors $e^{U^2/2-V^2/2}$.
Hence by dominated convergence, the integration contours $\Gamma_n'$ and $\Gamma_n$ can be extended to infinity in the Fredholm determinant
without changing the limit on the right-hand side of \eqref{deteqTaylor}.
The integration contours for $U$ and $V$ can be deformed to $1+i\R$ and to $\C_0$ respectively by Cauchy's integral theorem.

Finally we reformulate the kernel as follows.
Since $\Re(U-V)>0$ for $U\in1+i\R$ and $V\in\C_0$, we have that
\begin{equation}
\frac1{U-V}=\int_{\R_+}e^{-x(U-V)}\d x.
\end{equation}
Hence we can write the kernel with the contours extended to infinity as
\begin{equation}
K'_{s,\infty}(U,U')=-AB(U,U')
\end{equation}
where
\begin{equation}
A(U,x)=e^{U^2/2-(h+x)U-s/U},\qquad B(x,U)=\frac1{2\pi i}\int_{\C_0}e^{-V^2/2+(h+x)V+s/V}\frac{\d V}{V(V-U)}.
\end{equation}
Since $BA(x,y)=K_s(x,y)$, we conclude \eqref{K'conv} by using the fact that $\det(\id-AB)_{L^2(1+i\R)}=\det(\id-BA)_{L^2(\R_+)}$.
\end{proof}

\section{Tracy--Widom limit}
\label{s:TW}

In this section we prove Theorem~\ref{thm:HstoTW} and Corollaries~\ref{cor:clusterwidth} and \ref{cor:passagetimetoTW}
about the Tracy--Widom limit of $H_s$ as well as its consequences on the height of the percolation cluster.

\begin{proof}[Proof of Theorem~\ref{thm:HstoTW}]
We introduce the scaling of the space variables given by $x=2^{-2/3}3s^{1/3}+2^{-4/9}3^{1/3}s^{-1/9}X$
and $y=2^{-2/3}3s^{1/3}+2^{-4/9}3^{1/3}s^{-1/9}Y$
and we apply the change of variables $u=2^{1/3}s^{1/3}+2^{4/9}3^{-1/3}Us^{1/9}$ and $v=2^{1/3}s^{1/3}+2^{4/9}3^{-1/3}Vs^{1/9}$ in \eqref{defKs}.
In the exponent after using the identity $1/(1+q)=1-q+q^2-q^3/(1+q)$ the linear and quadratic terms in $U$ and $V$ cancel and we get that
\begin{equation}\label{exponentcalc}
\frac{u^2}2-yu-\frac su=-2^{-1/3}3s^{2/3}-2^{-1/9}3^{1/3}s^{2/9}Y+\frac{U^3}3\frac1{1+2^{1/9}3^{-1/3}s^{-2/9}U}-UY
\end{equation}
and a similar identity in $v$ and $x$.
This means that the rescaled kernel after a conjugation is equal to
\begin{equation}\label{Airyconv}\begin{aligned}
&e^{2^{-1/9}3^{1/3}s^{2/9}(X-Y)}\\
&\quad\times2^{-4/9}3^{1/3}K_s\left(2^{-2/3}3s^{1/3}+2^{-4/9}3^{1/3}s^{-1/9}X,2^{-2/3}3s^{1/3}+2^{-4/9}3^{1/3}s^{-1/9}Y\right)\\
&\qquad=\frac1{(2\pi i)^2}\int\d U\int\d V\frac{e^{\frac{U^3}3\frac1{1+2^{1/9}3^{-1/3}s^{-2/9}U}-UY-\frac{V^3}3\frac1{1+2^{1/9}3^{-1/3}s^{-2/9}V}+VX}}{V-U}+o(1).
\end{aligned}\end{equation}
The integration contours for $U$ and $V$ can be obtained as follows.
We first deform the original contours for $u$ and $v$ in \eqref{defKs} so that they pass through $2^{1/3}s^{1/3}$.
We may choose the contour for $u$ to be the $V$ shaped contour $V_{2^{1/3}s^{1/3},\pi/2-\varepsilon}^\infty$
and the contour for $v$ to be a circle of radius $2^{1/3}s^{1/3}$
which is deformed locally so that it coincides with $V_{2^{1/3}s^{1/3},\pi/2+\varepsilon}^\infty$ around $2^{1/3}s^{1/3}$ for some small fixed $\varepsilon>0$.
We claim that the contour for $U$ on the right-hand side of \eqref{Airyconv} can be chosen to be the one
which follows the semi-infinite straight lines from $e^{-i(\pi/2-\varepsilon)}\infty$ to $0$ and from $0$ to $e^{i(\pi/2+\varepsilon)}\infty$
and the contour for $V$ can be the one which goes from $e^{-i(\pi/2+\varepsilon)}\infty$ to $0$ and from $0$ to $e^{i(\pi/2+\varepsilon)}\infty$.
The fact that the two contours intersect at $0$ does not cause divergence, alternatively it can be avoided by local deformation.

To validate the choice of contours described above we prove that the integrand has enough decay so that the integral in $U$ and $V$ can be localized
to a small neighbourhood of $2^{1/3}s^{1/3}$.
In order to justify the localization we first prove that if $U=e^{i(\pi-\varepsilon)}t$ and $t\ge0$ then for $s\ge2^{1/2}3^{-3/2}$ it holds that
\begin{equation}\label{Rebound}
\Re\left(\frac{U^3}3\frac1{1+2^{1/9}3^{-1/3}s^{-2/9}U}\right)\le-\frac{\sin(3\varepsilon)}3\frac{t^2}{1+t}.
\end{equation}
To see \eqref{Rebound} we observe that the argument $\arg(U^3/3)=3\pi/2-3\varepsilon$ and that
$\arg(1+2^{1/9}3^{-1/3}s^{-2/9}U)\in[0,\pi/2-\varepsilon]$ for all $t\ge0$.
On the other hand $|U^3/3|=t^3/3$ and for $s\ge2^{1/2}3^{-3/2}$ we have that $|1+2^{1/9}3^{-1/3}s^{-2/9}U|\le1+t$.
This shows that for the complex number $z=U^3/(3(1+2^{1/9}3^{-1/3}s^{-2/9}U))$ it holds that $\arg(z)\in[\pi-2\varepsilon,3\pi/2-3\varepsilon]$
and $|z|\ge t^3/(3(1+t))$ hence its real part satisfies $\Re(z)\le-\sin(3\varepsilon)t^3/(3(1+t))$ proving \eqref{Rebound}.

The bound on the real part of the exponent given in \eqref{Rebound} and its analogue for $V$ proves
that the integrand on the right-hand side of \eqref{Airyconv} has at least Gaussian decay in $U$ and $V$
hence the error caused by changing the contours to be the ones given above causes an error going to $0$.
The integrand on the right-hand side of \eqref{Airyconv} converges for any $U$ and $V$
so that the double integral formally goes to the Airy kernel.
By the bound \eqref{Rebound} the Gaussian decay of the integrand is enough to conclude the convergence of the kernel.

For the convergence of the Fredholm determinants we can write
\begin{equation}\label{lambdaint}
\frac1{U-V}=\int_0^\infty\d\lambda e^{-\lambda(U-V)}
\end{equation}
because $\Re(U-V)>0$.
Using \eqref{lambdaint} on the right-hand side of \eqref{Airyconv} factorizes the integrand into $U$ and $V$ dependent parts.
Each of them has an Airy decay in $X$ and $Y$ which can be seen in the same way as for the Airy function.
The contour for $U$ can be deformed to coincide with the vertical line at $\sqrt{Y+\lambda}$ around the real axis
and to have $\Re(U)\ge\sqrt{Y+\lambda}$ along the whole contour.
Then $\Re(U^3/3-U(Y+\lambda))\le-\frac23(Y+\lambda)^{3/2}$ which yields the Airy decay and the convergence of the Fredholm determinants.
\end{proof}

\begin{proof}[Proof of Corollary~\ref{cor:clusterwidth}]
We can write the definition \eqref{defHst} as
\begin{equation}\label{Hstcomp}
H_s(t)=\lim_{n\to\infty}\sqrt t\sqrt{\frac{a+b}{ab}}\frac1{\sqrt{tn}}H\left(tn,\frac{\sqrt ts}{\sqrt{ab(a+b)}}(tn)^{-1/2}\right)\stackrel\d=\sqrt tH_{\sqrt ts}
\end{equation}
using Theorem~\ref{thm:BEasymptotics} in the second equality in distribution.
For any fixed $s>0$ Theorem~\ref{thm:HstoTW} with $s$ replaced by $\sqrt ts$ implies that the right-hand side of \eqref{Hstcomp} can be written as
\begin{equation}
\sqrt tH_{\sqrt ts}=\sqrt t\left(2^{-2/3}3(\sqrt ts)^{1/3}+2^{-4/9}3^{1/3}(\sqrt ts)^{-1/9}\xi_t\right)
\end{equation}
where $\xi_t$ converges in law to the Tracy--Widom distribution which proves \eqref{Hstconv}.
\end{proof}

\begin{proof}[Proof of Corollary~\ref{cor:passagetimetoTW}]
The statement of Theorem~\ref{thm:HstoTW} in terms of the Fredholm determinant in \eqref{Hsdistribution} means that
by setting $h(s)=2^{-2/3}3s^{1/3}+2^{-4/9}3^{1/3}s^{-1/9}r$ for any $r\in\R$ we have $\det(\id-K_s)_{L^2((h(s),\infty))}\to\P(\xi<r)$ as $s\to\infty$.
We express the convergence of the Fredholm determinant using the variable $h$ as follows.
We introduce $s(h)=4h^3/27-(2/3)^{4/3}h^{5/3}r$ which has the property that $h(s(h))=h+\O(h^{-5/3})$ where the error is of smaller order than the fluctuations.
Hence we can write $\det(\id-K_{s(h)})_{L^2((h,\infty))}\to\P(\xi<r)$ as $h\to\infty$.
By \eqref{Thdistribution} this implies \eqref{passagetimetoTW}.
\end{proof}

\section{Decay bounds}
\label{s:decaybounds}

In this section we prove the decay bounds in Proposition~\ref{prop:Hstail}.

\begin{proof}[Proof of Proposition~\ref{prop:Hstail}]
By Cauchy's integral theorem, the integration contours in the definition \eqref{defKs} of the kernel $K_s$ can be deformed as long as no singularity is crossed
and the decay along the infinite contour is guaranteed during the deformation.
Our choice is $K+i\R$ for the variable $u$ and the circle of radius $\varepsilon$ around $0$ for $v$
with the values of $K$ and $\varepsilon$ to be specified later so that $K>\varepsilon>0$.
Writing $u=K+it$ with $t\in\R$ one observes that $\Re(u^2)=K^2-t^2$, hence by bounding the absolute value of each factor of the kernel, we have that
\begin{equation}\label{Ksbound}
|K_s(x,y)|\le C\varepsilon\int_\R\d t\,e^{K^2/2-t^2/2-Ky+s/K+\varepsilon^2/2+\varepsilon x+s/\varepsilon}\frac{\sqrt{K^2+t^2}}\varepsilon\frac1{K-\varepsilon}.
\end{equation}
In the formula above and later in this proof $C$ denotes a finite positive constant which may change from line to line.

The integration in $t$ can be performed after using the inequality $\sqrt{K^2+t^2}\le K+|t|$ as
\begin{equation}\label{tintbound}
\int_\R\d t\,e^{-t^2/2}\sqrt{K^2+t^2}\le C(1+K).
\end{equation}
If $K\ge1$, then the integral above can be bounded by $CK$.
Based on \eqref{Ksbound} and using \eqref{tintbound} we get that
\begin{equation}\label{Ksbound2}
|K_s(x,y)|\le Ce^{K^2/2-Ky+\varepsilon^2/2+\varepsilon x+s/K+s/\varepsilon}.
\end{equation}

By the Fredholm expansion on the right-hand side of \eqref{Hsdistribution}, we have that the tail probability of $H_s$ can be written as
\begin{equation}
\P(H_s>h)=\sum_{m=1}^\infty\frac{(-1)^{m+1}}{m!}\int_h^\infty\d x_1\dots\int_h^\infty\d x_m\det(K_s(x_i,x_j))_{i,j=1}^m.
\end{equation}
Using \eqref{Ksbound2} and Hadamard's inequality on the $m\times m$ determinant above, we get that
\begin{equation}\label{Hstailbound}\begin{aligned}
\P(H_s>h)&\le\sum_{m=1}^\infty\frac{m^{m/2}}{m!}C^me^{m(K^2/2+\varepsilon^2/2+s/K+s/\varepsilon)}
\int_h^\infty\d x_1\dots\int_h^\infty\d x_m\,e^{-(K-\varepsilon)\sum_{i=1}^mx_i}\\
&=\sum_{m=1}^\infty\frac{m^{m/2}}{m!}\left(\frac{Ce^{K^2/2-(K-\varepsilon)h+\varepsilon^2/2+s/K+s/\varepsilon}}{K-\varepsilon}\right)^m.
\end{aligned}\end{equation}

The values of $K$ and $\varepsilon$ are to be chosen in a way that we get the best bound in \eqref{Hstailbound}.
With $K=h$, the expression $K^2/2-Kh$ in the exponent on the right-hand side of \eqref{Hstailbound} is minimized and its value is $-h^2/2$.
If $0\le s\le h$, then the term $s/K$ in the exponent is bounded by $1$.
In this case, we choose $\varepsilon=\sqrt{s/h}$ which minimizes the term $\varepsilon h+s/\varepsilon$ in the exponent with minimal value $2\sqrt{sh}$.
With this choice of $K$ and $\varepsilon$ each of the terms $s/K$ and $\varepsilon^2/2$ in the exponent
is upper bounded by $s/h$.
If we choose $h_0>1$ which means together with the condition $s\le h$ that $s<h^3$ then $s/h<\sqrt{sh}$ holds.
This also guarantees that $\varepsilon<K$ and the two integration contours do not cross.
Hence we get that
\begin{equation}
\P(H_s>h)\le\sum_{m=1}^\infty\frac{m^{m/2}}{m!}\left(\frac{Ce^{-h^2/2+4\sqrt{sh}}}h\right)^m
\end{equation}
where the $m=1$ term gives the desired upper bound on the right-hand side of \eqref{Hstail}
and further terms are negligible compared to it if $h$ is large enough.
This proves the first part of the proposition.

If $s\to\infty$ with $h$ satisfying $s\ll h^3$, then we again choose $\varepsilon=\sqrt{s/h}$ and $K=h$.
We have that $\varepsilon h+\varepsilon^2/2+s/K+s/\varepsilon=(2+o(1))\sqrt{sh}$ and the rest of the proof is the same as in the first case.

If $s=ch^3$, then we choose $K=\kappa(c)h$ and $\varepsilon=e(c)h$
where $\kappa(c)$ and $e(c)$ minimize the expressions $\kappa(c)^2/2-\kappa(c)+c/\kappa(c)$ and $e(c)^2/2+e(c)+c/e(c)$
which appear as the coefficients of the $h^2$ term in the exponent of \eqref{Hstailbound}.
By taking the derivative we solve the equations $\kappa(c)-1-c/\kappa(c)^2=0$ and $e(c)+1-c/e(c)^2=0$ which have exactly one positive solution for $c>0$.
These solutions denoted by $\kappa(c)$ and $e(c)$ satisfy $\lim_{c\to0}\kappa(c)=1$ and $\lim_{c\to0}e(c)=0$ with $\lim_{c\to0}e(c)/\sqrt c=1$.
Hence the sum of $\kappa(c)^2/2-\kappa(c)+c/\kappa(c)$ and $e(c)^2/2+e(c)+c/e(c)$ is negative for all $c\in(0,c^*)$ with some $c^*>0$.
Furthermore $e(c)<\kappa(c)$ also holds on this interval so that the two contours do not cross.
Numerical approximation yields that $c^*\simeq0.0468$.
\end{proof}

\bibliography{bernoulli_exp}
\bibliographystyle{alpha}
\end{document}